\documentclass{amsart}

\usepackage{tikz,amssymb,amsbsy,amsthm,amsmath,graphicx,epsfig,color}

\newtheorem{thm}{Theorem}[section]
\newtheorem*{thm*}{Theorem}

\newtheorem{lemma}[thm]{Lemma}
\newtheorem{prop}[thm]{Proposition}

\newtheorem{corollary}[thm]{Corollary}

\newtheorem{fact}[thm]{Fact}

\theoremstyle{definition}
\newtheorem{dfn}[thm]{Definition}
\newtheorem*{dfn*}{Definition}

\newtheorem*{ques*}{Questions}
\newtheorem{ex}[thm]{Example}
\newtheorem{remark}[thm]{Remark}

\newcommand{\radius}{\mathrm{r}}
\newcommand{\Id}{\mathrm{Id}}

\newcommand{\N}{\mathbb{N}}

\newcommand{\R}{\mathbb{R}}
\newcommand{\Q}{\mathbb{Q}}
\newcommand{\C}{\mathbb{C}}

\newcommand{\Z}{\mathbb{Z}}

\newcommand{\lin}{\mathrm{lin}}
\newcommand{\dist}{\mathrm{dist}}
\newcommand{\rg}{\mathrm{rg}}

\newcommand{\ol}[1]{\overline{#1}}

\renewcommand{\Re}{\operatorname{Re}}

\newcommand{\codim}{\operatorname{codim}}

\title{Embeddability of real and positive operators}

\author{Tanja Eisner}
\address{Institute of Mathematics, University of Leipzig,
P.O. Box 100 920, 04009 Leipzig, Germany}
\email{eisner@math.uni-leipzig.de}

\author{Agnes Radl}
\address{Institute of Mathematics, Technical University of Berlin,
Strasse des 17. Juni 136,
10623 Berlin, Germany}
\email{radl@math.tu-berlin.de}



\begin{document}

\maketitle

\begin{center}
\emph{\small Dedicated to Rainer Nagel on the occasion of his 80$^{\text{th}}$ birthday}
\end{center}

\begin{abstract}
Embedding discrete Markov chains into continuous ones is a famous open problem in probability theory with many applications. 
Inspired by recent progress, we study the closely related questions of embeddability of real and positive operators into real or positive $C_0$-semigroups, respectively, on finite and infinite-dimensional separable sequence spaces.  For the real case we give both sufficient and necessary conditions for embeddability. For positive operators we present necessary conditions for positive embeddability including a full  description for the $2\times 2$-case. Moreover, we show that real embeddability is topologically typical for real contractions on $\ell^2$. 
\end{abstract}

\section{Introduction}
The problem of embedding a discrete object into a continuous one with the same properties appears in many forms. A famous 
problem in this context is finding a Markov semigroup $(T(t))_{t\geq 0}$ for a given Markov matrix $T$ such that $T(1)=T$, see, 
e.g., \cite{elfving,kingman,davies,BS} 
as well as \cite{veerman-kummel} for a connection to graph theory.
 This question has a wide range of applications, e.~g.~in sociology, biology and finance, see \cite{socio, bio, finance}. 
For the ergodic theoretic setting of 
embedding a measure-preserving transformation into a measure-preserving flow
see \cite{king,rue-lazaro,step-erem}. An analogous question in stochastics and measure theory was considered in \cite[Chapter III]{heyer} and \cite{fischer}.  The  operator theoretic setting of embedding an operator into a $C_0$-semigroup was discussed in \cite{haase-book,haase,eis}. Moreover, see \cite{physics, quantum}  for an analogous question in quantum information theory.

In this paper we study 
embeddability in the context of real and positive bounded linear operators. Here, \emph{real} operators are operators on  complexifications of real Banach spaces which map real vectors to real vectors. We call a  strongly continuous  semigroup  (or $C_0$-semigroup for short) $(T(t))_{t\geq 0}$ or a group $(T(t))_{t\in\R}$ \emph{real} if every operator $T(t)$ is real. Analogously, a \emph{(strictly) positive} operator on a Banach lattice is an operator which maps positive (nonzero) vectors into (strictly) positive ones. 
A strongly continuous semigroup $(T(t))_{t\geq 0}$ or a group $(T(t))_{t\in\R}$ is called \emph{(strictly) positive} if $T(t)$ is (strictly) positive for each $t\neq 0$. For the general theory of $C_0$-semigroups see, e.g., \cite{EN}.

\begin{dfn*}
We call a real operator $T$
\emph{real-embeddable} if there exists a real strongly continuous  
semigroup $(T(t))_{t\geq 0}$ with $T=T(1)$. Analogously, we call a positive operator $T$ on a Banach lattice \emph{positively embeddable} if there exists a positive strongly continuous 
semigroup $(T(t))_{t\geq 0}$ with $T=T(1)$.
\end{dfn*} 
Analogously, one defines real or positive embeddability into a (strongly continuous) group. 

We study in this paper the following questions. 
\begin{ques*}\label{ques}
\begin{itemize}
\item[1)] When is a real operator real-embeddable? 
\item[2)]
When is a positive operator positively embeddable? 
\end{itemize}
\end{ques*}
Clearly, these two questions are closely related and also related to the question when a Markov matrix/operator is embeddable into a Markov semigroup.

A natural idea to embed $T$ is to define a logarithm $A:=\log T$ using  some functional calculus provided that certain spectral properties of $T$ are satisfied. In this case, $A$ generates a strongly continuous semigroup $(T(t))_{t\geq 0}$ with $T(1)=T$. Alternatively, one can use functional calculus to define the semigroup directly by $T(t):=T^t$. However, the semigroup operators need not satisfy the additional requirements to be real or positive for all $t\geq 0$. 
Moreover, if $T$ is embeddable into a strongly continuous semigroup, the generator (or the semigroup) need not be of the form provided via some functional calculus. 
Note also that an embedding into a semigroup is not necessarily unique, not even in the case of Markov matrices, see  \cite{speakman}.

Even though the statement of the embedding problem is rather simple,  in the Markov case only embedding of $d\times d$ matrices for $d\leq 4$ is well understood, see \cite{kingman,BS,cuthbert,johansen,carette,4x4}. 
For the positive case we present, among other results, a full characterisation in dimension $2$. Real embeddability in the finite-dimensional case was characterised in \cite[Thm.~1]{culver} and  \cite[Thm.~6.4.15]{horn-johnson}, see also Theorem \ref{thm:real} below.

Note that embeddability into a semigroup with a certain property automatically implies existence of the roots of all orders with this property.
 The problem of finding a stochastic root (of some order) of a stochastic matrix is known as the (discrete) embedding problem, see, e.~g., \cite{singer-spilerman7374,guerry2013,guerry2014}. The computation of stochastic roots of order $p$ is treated in \cite{higham-lin}. 

The  question, which positive matrices have positive roots of all  orders, was posed in \cite{johnson} in the study of so-called $M$-matrices. Recently, this study was continued in \cite{van-brunt} where  the above question 2) is  addressed  for finite positive  matrices, whereas the finite real case was treated in \cite{culver} and \cite{horn-johnson}. 
For aspects of complexity  of finding a positive square root of a finite positive matrix see  \cite{bausch-cubitt}.

While we investigate real embeddability for general bounded linear operators on Banach spaces, in our study of positive embeddability  we restrict our considerations to (finite and infinite) matrices. By an \emph{infinite matrix} $T$ we mean a bounded linear operator on any of the complex sequence spaces $c_0$, $\ell^p$, $1\leq p<\infty$, having $T$ as the representation with respect to the canonical unit vectors. Note that (strict) positivity in this context is equivalent to (strict) positivity of all entries of the matrix.

The paper is organised as follows. We first treat the real case in Section~\ref{sec:real}. Section~\ref{sec:pos-general} is devoted to positive embeddability of general finite and infinite matrices, while a complete answer in the $2\times 2$ 
case (following by a $3\times 3$ example) is  in Section~\ref{sec:2x2}. We finish the paper by showing that, in contrary to the finite-dimensional case,  real embeddability is typical for infinite matrices in Section~\ref{sec:typical}.

\textbf{Notation.}
 The set of real  $d\times d$ matrices is denoted by $\R^{d\times d}$ while the space of  bounded linear operators on a Banach space $X$ is denoted by  $\mathcal{L}(X)$. For the identity operator on $X$ we use the letter  $I$. The linear hull 
of vectors $x_1,\ldots, x_k\in X$ is  $\lin\{x_1,\ldots,x_k\}$.

  For an operator $T\in\mathcal{L}(X)$ we write $\sigma(T)$ for its  spectrum while $\rho(T)$ denotes its resolvent set. For  $\lambda\in\rho(T)$, the resolvent of $T$ at $\lambda$ is  $R(\lambda,T)$. Moreover, $\ker~T$ and $\rg~T$ are the kernel and the  range of $T$, respectively. We denote by $\codim(\rg~T):=\dim(X/\rg~T)$ the codimension of the range of $T$ in $X$. 

Finally, we use the standard notation $c_0$ and $\ell^p$, $1\leq p<\infty$, for the space of null sequences endowed with the supremum norm and the space of $p$-summable sequences with the $p$-norm, respectively.

\section{Preliminaries and real embeddability}\label{sec:real}

We recall the following  necessary condition for embeddability of an operator $T$ on a Banach space $X$ into a $C_0$-semigroup, i.e., the existence of a $C_0$-semigroup $(T(t))_{t\geq 0}$ on $X$ with $T=T(1)$, see \cite[Thm.~3.1]{eis}.

\begin{prop}\label{prop:spectrum-zero}
Let $X$ be a Banach space
 and $T\in\mathcal{L}(X)$. If $T$ can be embedded into a $C_0$-
semigroup, then both $\dim(\ker T)$ and $\codim(\rg T)
$
are equal to zero or to infinity.
\end{prop}

\subsection{Finite-dimensional case}\label{subsec:real-finite}

We first consider finite matrices and note that every embeddable finite matrix is necessarily invertible by Proposition \ref{prop:spectrum-zero}. 
 Moreover, every finite invertible real matrix $T$ having a real square root  $S$ satisfies $\det T =(\det S)^2>0$. This gives a simple necessary condition for real (and positive) embeddability in the finite-dimensional case.

We start with real embeddability of Jordan blocks being  essential for the general case treated in Theorem \ref{thm:real} below.
We observe in the following example that Jordan blocks to positive eigenvalues are real-embeddable. 
As we will see, the reason is that the function $z\mapsto z^t$ is holomorphic in a neighbourhood of every positive number and has real coefficients in the corresponding Taylor series (or, equivalently, is real if $z$ is real).
  In Remark~\ref{rem:dunford}~2) we give a general property of a function $f$ making $f(T)$ -- obtained via Dunford's functional calculus -- a real operator if $T$ is a real operator. 
  Note that Jordan blocks to negative eigenvalues are not real-embeddable, see Theorem~\ref{thm:real} below. 

\begin{ex}[Jordan blocks to positive eigenvalues]\label{ex:jordan-real}
Let $\lambda>0$, $d\in\N$, and consider the $d$-dimensional Jordan block 
$$
J:=\left(
\begin{matrix}
\lambda & 1& 0 & \ldots& 0\\
0 & \lambda & 1 &\ldots & 0\\
& & &\cdots\\
0& 0 & 0& \ldots & \lambda
\end{matrix}
\right).
$$
We show that $J$ is embeddable into a real group. 
Observe that 
$$J=\lambda \left(
\begin{matrix}
1 & \frac1\lambda& 0 & \ldots& 0\\
0 & 1 & \frac1\lambda &\ldots & 0\\
& & &\cdots\\
0& 0 & 0& \ldots & 1
\end{matrix}
\right),$$ so after a  change of basis with a real transformation matrix we can assume 
that $\lambda=1$. 
Writing $J= I +N$ if $\lambda =1$ and using that $N$ is nilpotent of index $d$ one obtains by virtue of Dunford's functional calculus a real group $T(\cdot)$ where 
$$
T(t):=(I+N)^t=\sum_{n=0}^{d-1} \binom{t}{n}N^n=\left(
\begin{matrix}
1 & t& \frac{t(t-1)}{2} & \ldots& \frac{t(t-1)\ldots(t-d+2)}{(d-1)!}\\
0 & 1 & t&\ldots &\frac{t(t-1)\ldots(t-d+3)}{(d-2)!} \\
& & &\cdots\\
0& 0 & 0& \ldots & 1
\end{matrix}
\right)
$$
for $t\in \R$ and $\binom{t}{n}:=\frac{t(t-1)\cdots(t-n+1)}{n!}$. 
Indeed, $T(0)=I$ and $T(1)=J$ are clear while the property $T(t+s)=T(t)T(s), t, s\in \R$, follows from the identity $z^{t+s}=z^{t}z^{s}$ for $z>0$ and $t, s\in\R$. 

Note that this formula coincides with the one in  \cite[Def.~1.2]{higham-book}.
One can already guess it without invoking Dunford's functional calculus by computing the powers of $J=I+N$ with the binomial theorem and then extrapolating for $t\not\in\mathbb{N}$. 
Here, it remains to show that the generalised binomial coefficients given by $a_j(t):=\binom{t}{j}$ 
satisfy the relation 
\begin{equation}\label{eq:binom}
a_j(t+s)=a_j(s)+a_1(t)a_{j-1}(s)+\ldots+a_{j-1}(t)a_1(s)+a_j(t)
\end{equation}
for every $j\in\N$ and $t,s\in\R$. But this is the well-known Chu-Vandermonde identity.

In Example \ref{ex:pos-Jordan} below we will see that $J$ is embeddable into a positive semigroup if and only if $d\leq 2$. 
\end{ex}

We will use the following classical fact from linear algebra.
\begin{fact}\label{fact}
Let $J$ be a Jordan block corresponding to $\lambda\in\C$ with $\lambda\neq 0$. Then $J^2$ is up to a change of basis a Jordan block corresponding to $\lambda^2$.
\end{fact}
To see this, let $d$ be the dimension of $J$ and write $J=\lambda I + N$, where $N$ is nilpotent of index $d$. So $J^2=\lambda^2 I + N(2\lambda I+N)$. Since $\lambda \neq 0$, $N(2\lambda I+N)$ is nilpotent of index $d$, and therefore $J^2$ is up to a change of basis a Jordan block corresponding to $\lambda^2$.
Note that for $\lambda=0$ the assertion is false already for $d=2$.

The following result is known, see  \cite[Thm.~1]{culver} and  \cite[Thm.~6.4.15]{horn-johnson}. 
For the reader's convenience we present here a direct proof.

\begin{thm}[Characterisation of real embeddability in finite dimensions]\label{thm:real}
  Let $T$ be a real $d\times d$ matrix with $0\notin\sigma(T)$. 
Then the following assertions are equivalent.
\begin{enumerate}
\item[(i)] $T$ is embeddable into a real semigroup.
\item[(ii)] $T$ has a real square root.
\item[(iii)] In the Jordan normal form of $T$ every Jordan block to every negative eigenvalue appears evenly many times.
\end{enumerate} 
\end{thm}
\begin{proof}
We begin with an observation. Let $S$ be a real matrix with eigenvalue $\lambda\in\C\setminus\R$. Consider  a Jordan block of dimension $k$ to $\lambda$ with corresponding Jordan chain 
$x_1,\ldots,x_k$.
Since $S$ is real,  $\ol{x_1},\ldots,\ol{x_k}$  is a Jordan chain
of $S$ corresponding to $\ol{\lambda}$. 
By Fact \ref{fact}, 
$S^2$ has two Jordan blocks  of the same dimension, one corresponding to $\lambda^2$ and one to $\ol{\lambda}^2$.

We now come to the proof of the assertion. 
Clearly, (i)$\Rightarrow$(ii). Moreover, (ii)$\Rightarrow$(iii) follows from 
the observation above for a real square root $S$ of $T$ and $\lambda\in \sigma(S)\cap i\R$ (if it exists).

(iii)$\Rightarrow$(i) Consider the Jordan normal form of $T$. For a $k$-dimensional Jordan block $J$ with an eigenvalue $\lambda$ denote by $x_1,\ldots,x_k$ the corresponding Jordan chain.
Let 
$H(J):=\lin\{x_1,\ldots,x_k\}$.
Note that if $\lambda$ is real, then the Jordan chain
can be chosen to be real. If $\lambda\in\C\setminus\R$, then by the above observation 
$T$ has a Jordan block $\ol{J}$ to $\ol{\lambda}$ of the same dimension  with Jordan chain
$\ol{x_1},\ldots,\ol{x_k}$. Note that 
$$
H(J)\oplus H(\ol{J})=\lin\{x_1+\ol{x_1},\ldots,x_k+\ol{x_k},i(x_1-\ol{x_1}),\ldots,i(x_k-\ol{x_k})\},
$$
 where the vectors on the right hand side are real and linearly independent.  We will embed  $T$ restricted to $H(J)$ for real eigenvalues and to $H(J)\oplus H(\ol{J})$ for nonreal eigenvalues  into a real semigroup separately.

We begin with $\lambda\in\C\setminus\R$ and embed $T$ on $H(J)$ into  some, not necessarily real, semigroup $S(\cdot)$ (using for example the logarithm). Define $S(t)\ol{x_j}:=\ol{S(t)x_j}$ for every $j\in\{1,\ldots,k\}$ and $t\geq 0$ and extend $S(\cdot)$ linearly to $H(J)\oplus H(\ol{J})$.  

We claim that the obtained semigroup on $\lin\{x_1,\ldots,x_k,\ol{x_1},\ldots,\ol{x_k}\}$ is real. Indeed, every $S(t)$ has real  values on linearly independent real vectors $x_1+\ol{x_1},\ldots,x_k+\ol{x_k},i(x_1-\ol{x_1}),\ldots,i(x_k-\ol{x_k})$.

To embed $T$ on $H(J)$ for real $\lambda$ we can assume without loss of generality that $T=J$, since the  change of basis is provided by a real matrix for such $J$.
By  Example~\ref{ex:jordan-real}, Jordan blocks to positive eigenvalues are embeddable into real semigroups.
It remains now to embed Jordan blocks corresponding to negative 
 eigenvalues, which by (iii) come in pairs in the Jordan normal form of $T$. To do this, take such a Jordan block $J$.  
The real matrix 
$
S:=\left(
\begin{matrix}
0 & I\\
J & 0
\end{matrix}
\right)
$
written in the block form satisfies 
$S^2=\left(
\begin{matrix}
J & 0\\
0 & J
\end{matrix}
\right)$
and $\sigma(S)\subset i\R\setminus\{0\}$  by the spectral mapping theorem, so
has in particular no negative eigenvalues. So by the above we can embed $S$ and hence $\left(
\begin{matrix}
J & 0\\
0 & J
\end{matrix}
\right)$
into a real semigroup. The proof is complete.
\end{proof}
For an algorithm of constructing a real square root see  \cite{higham}.

\smallskip

\subsection{Infinite-dimensional case}
The rest of this section is devoted to the extensions of 
the characterisation in Theorem \ref{thm:real} to 
infinite-dimensional spaces where the situation is more complex. 

 We first recall some basic definitions of complexifications of real Banach spaces, see, e.g.,  \cite{MST} or  \cite[§II.11]{schaefer} for details. Let $Y$ be a real Banach space and 
let $X:=Y+iY$ be its complexification
endowed with the norm 
$$
\|y_1+iy_2\|:=\sup_{t\in[0,2\pi]}\| \cos t\cdot y_1 - \sin t \cdot
y_2\|
$$ turning $X$ into a Banach space.  
The dual space of $X$ is of the form $X'=Y'+iY'$. We call an operator $T$ on $X$ \emph{real} if $TY\subset Y$. Moreover, for an operator $T$ on $X$ we define its complex conjugate by $\overline{T}x:=\overline{T\ol{x}}$, where $\ol{y_1+iy_2}:=y_1-iy_2$. It is easy to verify that $\ol{T+S}=\ol{T}+\ol{S}$, $\ol{TS}=\ol{T}\cdot
\ol{S}$, $\lim \ol{T_n}=\ol{\lim T_n}$ and $\ol{T'}=(\ol{T})'$. Finally, $T$ is real if and only if $T=\ol{T}$, $T'$ is real whenever $T$ is, and the spectrum of a real operator is invariant under complex conjugation. 

We first generalise the 
sufficient spectral condition for Markov matrices given in  \cite[Prop.~1]{davies}. Recall that an operator $T$ is called \emph{sectorial} if there exists a sector $\Sigma_\delta:=\{z\in\C:\,|\arg z| \leq \delta\}\cup\{0\}$ with $\delta\in (0,\pi)$ such that $\sigma(T)\subset \Sigma_\delta$ and $\sup_{\lambda\notin\Sigma_\delta}\|\lambda R(\lambda,T)\|<\infty$. 
 Equivalently, $T$ is sectorial if 
$(-\infty,0)\subset\rho(T)$
and $\sup_{s<0}\|sR(s,T)\|<\infty$, see, e.g., \cite[Prop.~2.1.1]{haase-book}.

\begin{thm}[Sufficient condition for real embeddability]\label{thm:real-ln}
\begin{enumerate}
\item Let $T$ be a  real operator 
with $\sigma(T)\subset \C\setminus (-\infty,0]$. Then $T$ is embeddable into a real group with bounded generator. 
\item Let $T$ be a real sectorial  operator with dense range.  Then $T$ is real-embeddable.
\end{enumerate}
\end{thm}
\begin{proof}
(a) 
Let $\log$ denote the principal value of the complex logarithm. Note that $\log \ol{z}=\ol{\log z}$ for every $z\in \C\setminus (-\infty,0]$.

Denote $r:=\frac{\dist(\sigma(T), (-\infty,0])}{2}>0$. Denote by $\gamma:[-1,1]\to\C$ a positively directed path around $\sigma(T)$ beginning at $2\|T\|$ whose image consists of three parts:
\begin{itemize}
\item[(i)]  a half circle with radius $r$ around $0$ in the right half plane, 
\item[(ii)] an almost full circle with radius $2\|T\|$ around $0$ up to points with negative real part and imaginary part in $(-r,r)$,
\item[(iii)] two horizontal lines connecting the two partial circles. 
\end{itemize}


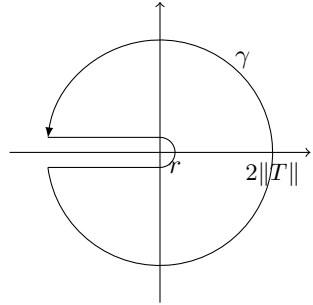
\begin{center}
\setlength{\unitlength}{1cm}
\thicklines
\begin{tikzpicture}
 
\draw[->] (-2,0) -- (2,0);
\draw[->] (0,-2) -- (0,2);

\draw [domain=-90:90] plot ({0.2*cos(\x)}, {0.2*sin(\x)});
\draw (0,0.2) -- (-1.49,0.2);
\draw (0,-0.2) -- (-1.49,-0.2);
\draw[-latex]
  (-1.49,-0.2) arc (-172.5:172.5:1.5);

\draw (0.2,0)  
  node [below] {$r$};
\draw (1.5,0)    node [below] {\small $ 2\|T\|$};
\draw (1.1,1.0)  
  node [above] {$\gamma$};


\end{tikzpicture}
\end{center}


Moreover, we take $\gamma$ to be symmetric with respect to the real line with $\gamma(-t)=\ol{\gamma(t)}$, $t\in[-1,1]$. Define now the bounded operator 
$$
A:=\frac1{2\pi i}\int_\gamma R(\lambda,T)\log \lambda \, d\lambda
$$
which generates the uniformly continuous group $(e^{tA})_{t\in\R}$.
By Dunford's functional calculus, we have $e^A=T$. 

We show first that $A$ is real. Indeed, for every $\lambda\in\rho(T)$, 
$$
R(\ol{\lambda},T)=(\ol{\lambda-T})^{-1}=\ol{R(\lambda, T)}
$$
and therefore 
$$
\ol{A}=-\frac1{2\pi i}\int_\gamma R(\ol{\lambda},T)\log \ol{\lambda} \, d\ol{\lambda} =-\frac1{2\pi i}\int_{-\gamma} R(\lambda,T)\log \lambda \, d\lambda =A,
$$
showing that $A$ is a real operator. By
$$
e^{tA}=\sum_{n=0}^\infty \frac{t^nA^n}{n!},
$$
the operator $e^{tA}$ is real for every $t\in\R$.
\medskip 

(b) By the functional calculus for sectorial operators, see \cite[Prop. 3.1.1 and 3.1.15]{haase-book}, $T$ can be embedded into an analytic semigroup $T(\cdot)$. Moreover, by the Balakrishnan representation \cite[Prop. 3.1.12]{haase-book}, 
$$
T(t)x=-\frac{\sin (t\pi)}{\pi}\int_0^\infty s^{t-1}R(-s,T)Tx \, ds
$$
holds for every $x\in X$ and every $t\in (0,1)$. Since $R(-s,T)$ is real as discussed in (a), $T(t)$ is real for every $t\in (0,1)$ and hence for every $t\geq 0$.
\end{proof}

\begin{remark}\label{rem:dunford}
\begin{itemize}
\item[1)]
Alternatively, one can define the (semi)group in the proof of Theorem \ref{thm:real-ln} directly by 
$$
T(t):= \frac1{2\pi i}\int_\gamma e^{t\log \lambda }R(\lambda,T) \, d\lambda
$$
for the same $\gamma$ as in the proof for part (a) and an appropriate modification of it for part (b) and show that it is real and, in the case of (a), uniformly continuous.
\item[2)]
 Inspecting the above proof one observes the following general property of Dunford's functional calculus: Let $U\subset \C$ be open and $f\in C(U)$  satisfy $f(\ol{z})=\ol{f(z)}$ for every $z\in U$, then $f(T)$ is real for every real operator $T$ with $\sigma(T)\subset U$. 
\end{itemize}
\end{remark}

The following example shows that the spectrum of a real-embeddable operator can be an arbitrary compact conjugation invariant subset of $\C$.
This shows that the sufficient condition in Theorem \ref{thm:real-ln} is far from being necessary (see also Theorem \ref{thm:shift} below). Moreover, a real operator which is not real-embeddable can have arbitrary compact  conjugation invariant set intersecting $(-\infty,0]$ as spectrum, see Example \ref{ex:real-spectrum-non-embed} below.

\begin{ex}[Spectrum of a real-embeddable operator]\label{ex:real-spectrum-embed}
Let $K\subset \C$ be a compact set with $\ol{K}=K$. We construct a real-embeddable operator $T$ with $\sigma(T)=K$. 
The construction is a natural modification of \cite[Example 2.4.]{eis}. 

Case 1: The point $0$ is not an  isolated point of $K$. Take a dense countable conjugation invariant subset of $K\setminus\{0\}$. Write this set as a sequence $(a_n)_{n=1}^\infty$, where $a_{2n-1}=\ol{a_{2n}}$ for every $n\in\N$, in particular, every negative number  appears twice.  

Consider, say, $X:=\ell^2$ with the standard (real) basis $(e_n)$ and define $(v_n)\subset X$ by $v_{2n-1}:=e_{2n-1}+ie_{2n}$ and $v_{2n}:=e_{2n-1}-ie_{2n}$ for every $n\in\N$. Then $(v_n)$ is clearly a (Schauder) basis of $X$. Define finally $T\in \mathcal{L}(X)$ by $Tv_n:=a_nv_n$. We first show that $T$ is real. Indeed, for every $n\in\N$
$$
Te_{2n-1}=\frac{a_{2n-1}v_{2n-1}+a_{2n}v_{2n}}{2}=e_{2n-1}\frac{a_{2n-1}+a_{2n}}2 + e_{2n}\frac{i(a_{2n-1}-a_{2n})}{2}
$$
which is a real vector by the construction of $(a_n)$. Analogously, $Te_{2n}$ is a real vector for every $n\in\N$ and thus $T$ is real. 

Define now $A\in\mathcal{L}(X)$ by $A v_n:=\log a_n\cdot v_n$, $n\in\N$, where $\log$ denotes the principal value of the logarithm if $a_n\notin (-\infty,0]$ and any other (fixed) branch of the logarithm otherwise. Then $A$ is real by the same reason as $T$ and therefore generates a real semigroup embedding $T$.

Case 2: The point $0$ is an isolated point of $K$. Let $T_1$ be the real operator constructed as above with spectrum $\sigma(T_1)=K\setminus\{0\}$. Consider $T_2=0$ on $\ell^2$ which is real-embeddable by Remark \ref{rem:zero-op-real} below. Thus $T_1\oplus T_2$ on $\ell^2\oplus \ell^2$ is 
real-embeddable with $\sigma(T)=K$. 
\end{ex}

\begin{remark}
As the construction in Example \ref{ex:real-spectrum-embed} (Case 1) shows, the operator $-I$ is real-embeddable on $\ell^p$ for every $1\leq p< \infty$. However, for every $p\notin\{1,2\}$, $-I$ is not embeddable into a real \emph{contractive} semigroup, see Remark \ref{rem:vera} below. 
\end{remark}

Inspired by Theorem \ref{thm:real}, we present a simple necessary condition for real embeddability in infinite dimensions. First we need some preparation motivated by the concept of Jordan chains from finite-dimensional linear algebra.

\begin{dfn}
We say that an operator $T$ on a Banach space $X$ has a \emph{(partial) Jordan block} of dimension $d\in\N$ corresponding to $\lambda\in\C$  if there exists a 
family of linearly independent vectors $x_1,\ldots,x_d\in X$ with $Tx_1=\lambda x_1$, $Tx_2=x_1+\lambda x_2$, $\ldots$, $Tx_d=x_{d-1}+\lambda x_d$.
By the existence of several Jordan blocks of $T$ we mean the existence of several such families which are linearly independent.
\end{dfn}

\begin{remark}\label{rem:Jordan}
As in linear algebra, an operator $T$ has a Jordan block of dimension $d$ to $\lambda$ if and only if 
$$
\ker((\lambda-T)^d)\setminus \ker((\lambda-T)^{d-1}) \neq \emptyset.
$$
Indeed, the family $x_1,\ldots,x_d$ defined by $x_d\in \ker((\lambda-T)^d)\setminus \ker((\lambda-T)^{d-1})$, $x_{d-1}:=(T-\lambda)x_d,\ldots, x_1:=(T-\lambda)x_2$ satisfies the requirement. 
\end{remark}

The following is a generalisation of the classical spectral mapping theorem for the point spectrum.

\begin{lemma}[Spectral mapping theorem for squares and Jordan blocks]\label{lem:SMT}
Let $S$ be an operator on a Banach space $X$, $d\in \N$ and $\lambda\in \C\setminus\{0\}$. Then $S^2$ has a Jordan block of dimension $d$ to $\lambda^2$ if and only if $S$ has a Jordan block of dimension $d$ to $\lambda$ or to $-\lambda$.
\end{lemma}
\begin{proof}
As in the classical case for $d=1$, we start with the identity
\begin{equation}\label{eq:square}
\lambda^2-S^2=(\lambda+S)(\lambda-S).
\end{equation}
For the ``only if'' direction, using Remark \ref{rem:Jordan}, assume that some $x\in X$ satisfies
\begin{equation}\label{eq:Jordan}
(\lambda^2-S^2)^dx=0,\  (\lambda^2-S^2)^{d-1}x\neq 0.
\end{equation}
There are several cases to consider. 

Case 1: $(\lambda+S)^dx=0$. Since \eqref{eq:square} and \eqref{eq:Jordan} imply $(\lambda-S)^{d-1}x\neq 0$, we are finished by Remark \ref{rem:Jordan}.

Case 2: $y:=(\lambda+S)^{d}x\neq 0$. By \eqref{eq:Jordan} we have
\begin{equation}\label{eq:product}
0=(\lambda-S)^dy=(\lambda+S)^d(\lambda-S)^dx.
\end{equation}

Case 2a: $(\lambda-S)^{d-1}y\neq 0$. Then we are finished by \eqref{eq:product} and Remark \ref{rem:Jordan}.

Case 2b: $(\lambda-S)^{d-1}y= 0$.  By the definition of $y$ this means 
$$
(\lambda+S)^d(\lambda -S)^{d-1}x=0.
$$
Defining $z:=(\lambda -S)^{d-1}x$ we obtain $(\lambda+S)^dz=0$
and, by \eqref{eq:Jordan}, $(\lambda+S)^{d-1}z\neq 0$, finishing the argument again by  Remark \ref{rem:Jordan}.
\smallskip

The ``if'' direction follows directly from Fact \ref{fact}.
\end{proof}

\begin{remark}
For $\lambda=0$, the ``only if'' direction in the above theorem holds with the same proof, whereas the ``if'' direction fails as the example $S=\begin{pmatrix} 0 & 1\\ 0&0  \end{pmatrix}$ shows. 
\end{remark}

The following result extends both Proposition \ref{prop:spectrum-zero} and Theorem \ref{thm:real}. 

\begin{thm}[Necessary condition for the existence of a real square root]\label{thm:real-spectrum-nec}
Let $T$ be a real operator 
with a real square root. If $T$ or $T'$ has a Jordan block of dimension $d$ to some $\lambda <0$,  then the number of such blocks is even or infinity. In particular, both $\dim \ker(\lambda-T)$ and $\codim\rg(\lambda-T)$ are zero, even or infinity for every $\lambda<0$. 
\end{thm}
\begin{proof}
Assume first that $T$ has a $d$-dimensional Jordan block to $\lambda=-\alpha^2<0$ and a real square root $S$. By Lemma \ref{lem:SMT}, $S$ has a $d$-dimensional Jordan block to  $i\alpha$ or $-i\alpha$. The rest follows as in the proof of  Theorem \ref{thm:real}.

The analogous assertion for $T'$ follows from the fact that $T'$ is also real and that $S'$ is a real square root of $T'$ whenever $S$ is one of $T$. 
\end{proof}

\begin{ex}[The operators $V-I$ and $S-I$]
\begin{itemize}
\item[1)]
Consider the real operator $T:=V-I$ on $C[0,1]$, where  $V:C[0,1]\rightarrow C[0,1], (Vf)(x):=\int_0^xf(t)\;\mathrm{d}t$ denotes the Volterra operator. By $\sigma(T)=\{-1\}$, $T$ is clearly embeddable into a uniformly continuous (semi)group by  virtue of  Dunford's functional calculus. However, $\codim\rg (T+I)=1$ implying that $-1$ is an eigenvalue of $T'$ with geometric multiplicity $1$. Therefore, by Theorem \ref{thm:real-spectrum-nec}, $T$ does not have a real square root and is hence not real-embeddable.

\item[2)] 
Consider the infinite matrix
$$
\left(
\begin{matrix}
-1 &0 & 0 & \ldots\\
1 & -1 & 0 & \ldots\\
0 & 1 & -1 & \ldots\\
&  \ldots & \\
\end{matrix}
\right)
$$
which
represents the operator $S-I$ for the right shift $S$ on $c_0$ or any $\ell^p$ with $1\leq p<\infty$ with respect to the canonical basis. This operator also does not have a real square root and is therefore  not real-embeddable by the same reasons as the operator in 1). Note that the above matrix represents the operator from 1) with respect to the real basis $(e_n)$ with $e_n(t):=\frac{t^n}{n!}$, $n\in \N_0$, $t\in [0,1]$.
\end{itemize}
\end{ex}

The following example shows that the 
set $\C\setminus (-\infty,0]$ in 
Theorem \ref{thm:real-ln} (a) cannot be enlarged. 

\begin{ex}[Spectrum of a real operator which is not real-embeddable]\label{ex:real-spectrum-non-embed}
Let $K\subset \C$ be a compact conjugation invariant set satisfying $K\cap(-\infty,0]\neq \emptyset$. We construct a real operator $T$ with $\sigma(T)=K$ which is not real-embeddable.

Case 1: $0$ is not an isolated point of $K$.
Let $T_1$ be a real operator on $\ell^2$
 with $\sigma(T_1)=K$ which is constructed by the same procedure as in Example \ref{ex:real-spectrum-embed}. 
Take $\lambda\in K\cap (-\infty,0]$ and define $X:=\ell^2\oplus \C$ and $T:=T_1\oplus \lambda$ which is clearly a real operator. Then $\dim\ker(\lambda-T)\in\{1,3\}$ (depending on whether $\lambda\in\{a_n:\,n\in\N\}$ or not). So $T$ is not real-embeddable by Theorem \ref{thm:real-spectrum-nec} or Proposition \ref{prop:spectrum-zero}, respectively. 

Case 2: $0$ is an isolated point of $K$, i.e., $K=K_1\cup \{0\}$ for a compact set $K_1\subset\C\setminus\{0\}$. Let $T_1$ be a real-embeddable operator on $X_1=\ell^2\oplus \ell^2$
with $\sigma(T)=K_1$ from Example \ref{ex:real-spectrum-embed}. Then the operator $T:=T_1\oplus 0$ on $X:=X_1\oplus \C$ is not embeddable by Proposition  \ref{prop:spectrum-zero}.
\end{ex}

\section{Positive embeddability: All dimensions}\label{sec:pos-general}

 In the following we assume that $T$ is a finite matrix or an infinite matrix on any of the sequence spaces $c_0$, $\ell^p$, $1\leq p<\infty$. 

\subsection{Zero pattern of positive semigroups}
In this subsection we study the zero pattern of positive semigroups and then derive necessary conditions for positive embeddability.
We begin with the following observation.
\begin{remark}\label{rem:entries-gen}
A finite or infinite matrix is positive if and only if all its entries are $\geq 0$, and strong convergence implies convergence of the entries. Therefore the generator $A=(b_{ij})$ of a positive semigroup with bounded generator satisfies by the definition of the generator 
\begin{equation*}
b_{ij}\begin{cases}
\in \R,\quad &i=j,\\
\geq 0 \quad &\text{otherwise}.
\end{cases}
\end{equation*}
\end{remark}

\begin{remark}\label{rem:vera}
\begin{itemize}
\item[1)] A necessary spectral condition for embeddability into a \emph{bounded} positive semigroup is the following result \cite[Thm.~3.1]{keicher}, cf.~ \cite[Thm.~9]{davies-05} and  \cite[Thm.~2.2]{wolff}: If a  power bounded positive  matrix operator $T$ is positively embeddable, then the only possible eigenvalue of $T$ on the unit circle is $1$. 
\item[2)] An analogous result also holds for real embeddability, see  \cite[Cor.~2.8]{glueck}: Let $T$ be a real operator on $\ell^p$ with $p\in(1,\infty)$, $p\neq 2$. If $T$ is embeddable into a real \emph{contractive} semigroup, then 
the only possible eigenvalue of $T$ on the unit circle is $1$. 
\end{itemize}
\end{remark}

The following provides a simple necessary condition of embeddability, cf.~the finite-dimensional result  \cite[Lemma 4 and Cor.~2]{van-brunt}. 
Recall that the class of analytic semigroups includes  semigroups with bounded generator, see e.g.~\cite[Section II.4.a]{EN}.

\begin{thm}[Zero entries]\label{thm:zeros}
Let $T=(a_{ij})$ be a positive finite or infinite 
 matrix which is embeddable into a positive semigroup $T(\cdot)=(a_{ij}(\cdot))$. Then the following holds. 
\begin{enumerate}
 \item $a_{jj}(t)> 0$ holds for every $t\geq 0$ and $j$. 
\item If $a_{ij}=0$ for some $i\neq j$, then $a_{ij}(t)=0$ for every $t\leq 1$. Moreover, if 
$T(\cdot)$ is analytic,  
then $a_{ij}(t)=0$ for every $t\in [0,\infty)$.
\end{enumerate}
\end{thm}
\begin{proof}
(a) Let $T(\cdot)=(a_{ij}(\cdot))$ be a positive finite or infinite matrix semigroup. Assume that $a_{jj}(t)=0$ for some $j$ and some $t>0$. By 
$$
0=a_{jj}(t)\geq a_{jj}(t/2)^2 \geq 0
$$
and induction we obtain $a_{jj}\left(\frac{t}{2^n}\right)=0$ for every $n\in\N$. This contradicts  the strong continuity of $T(\cdot)$ and $T(0)=I$.

(b) Assume that $T$ is embeddable into a positive semigroup $T(\cdot)$, $a_{ij}=0$ for some $i\neq j$ and let $t\in(0,1)$. Again by positivity we have 
$$
0=a_{ij}\geq a_{ij}(t) a_{jj}(1-t)\geq 0.
$$
Since by (a) $a_{jj}(1-t)>0$, we obtain $a_{ij}(t)=0$.

If in addition $T(\cdot)$ is analytic, then by the identity theorem 
the function $t\mapsto a_{ij}(t)$  equals zero on $[0,\infty)$.
\end{proof}

\begin{remark}
\begin{itemize}
\item[1)]
In view of Theorem \ref{thm:zeros} (b) it would be interesting to know whether there exists a positive $C_0$-semigroup $T(\cdot)$ of infinite matrices (necessarily not analytic) such that $T(1)_{ij}=0$ and $T(t)_{ij}\neq 0$ for some $t>1$ and some $i \neq j$.
\item[2)] Note that square roots of positive matrices need not have zero entries at the same places. Indeed, the matrix
$T:=\left(
\begin{matrix}
0 & 1\\
1 & 0
\end{matrix}
\right)
$
does not have zeros at the places where  $T^2=I$ has zeros. Note that this (or the failure of the condition (a) in Theorem \ref{thm:zeros}) implies that $T$ is not embeddable into a positive semigroup, although its square trivially is. Another example is the matrix $\left(
\begin{matrix}
1/2 & 1/2\\
1 & 0
\end{matrix}
\right)$ which is not positively embeddable by Theorem \ref{thm:zeros} (a), but its square is embeddable into a Markov semigroup, see \cite[Example 3.8]{BS}.
\end{itemize}
\end{remark}

The following easy consequence of Theorem \ref{thm:zeros}, a special case of \cite[Cor.~2.4]{wolff}, shows that the case of positive groups is rather simple. 
\begin{corollary}
Let $T$ be embeddable into a positive group. Then $T$ is a diagonal matrix with strictly positive diagonal.
\end{corollary}
\begin{proof}
Assume that $T=(a_{ij})$ satisfies $a_{ij}>0$ for some $i\neq j$ and that $T$ is embeddable into a positive group $(T(t))_{t\in\R}$ with $T(t)=(a_{ij}(t))$. By the semigroup law we have 
\begin{equation}\label{eq:pos-gr}
\sum_{k} a_{ik}(-1)a_{kj}= (T(-1)T)_{ij}=0.
\end{equation}
By Theorem \ref{thm:zeros} (a) 
and the assumption we see that $a_{ii}(-1)a_{ij}>0$, a contradiction to (\ref{eq:pos-gr}) and the positivity of the group. 
\end{proof}

\begin{remark} 
If $T$ is a diagonal matrix with strictly positive diagonal $(a_{jj})$, then it is embeddable into the positive group $(T(t))_{t\geq 0}$ of diagonal matrices with diagonal $\left(e^{t \log a_{jj}}\right)$, $t\geq 0$. Its generator $A$ satisfies $Ae_j=(\log a_{jj})e_j$. Note that $A$ need not be bounded as the example with $a_{jj}=e^{- j}$ shows.   
\end{remark}

Note that for real embeddability into groups the situation  is very different, namely, every bounded real operator $A$ generates a real group by virtue of the representation $T(t)=\sum_{n=0}^\infty \frac{t^nA^n}{n!}$.

\begin{corollary}
Let $T$ be a positively embeddable infinite  matrix. Then none of the basis vectors $e_j$ belongs to the kernel of $T$. In particular, the operator $T=0$ is not positively embeddable.
\end{corollary}

\begin{remark}\label{rem:zero-op-real}
  Note that $T=0$ on $\ell^2$ is embeddable into a real semigroup. Indeed, $\ell^2$ is isomorphic to $L^2[0,1]$ via an isomorphism preserving real vectors (e.g.\ by choosing an  orthonormal basis on $L^2[0,1]$ consisting of real trigonometric functions). Finally, $0$ is clearly embeddable into the real nilpotent shift semigroup on $L^2[0,1]$.
However, by the above corollary, it is not positively embeddable. 
\end{remark}

We now present an  {\em invertible}  infinite matrix that is real-embeddable, has positive roots of all orders but is not positively embeddable. This example is based on \cite[p.~24]{kingman} where it is used in the context of Markov chains.
\begin{ex}
  For a bijection $\sigma: \N\rightarrow\Q$ define $T$ as the infinite matrix $(a_{ij})$
with
  \[a_{ij}=\left\{\begin{array}{ll}1,&\sigma(i)+1=\sigma(j), \\0, &\text{else}.\end{array}\right.\]
Then $T$ is an invertible isometry on $c_0, \ell^p, 1\leq p<\infty$, and is by construction isomorphic via a positive isomorphism to the shift by $1$ to the left on $c_0(\Q), \ell^p(\Q), 1\leq p<\infty$, respectively. 
      A positive $n$-th root of $T$ is given by $R_n=\left(r^{(n)}_{ij}\right)$ where 
      \[r^{(n)}_{ij}=\left\{\begin{array}{ll}1,&\sigma(i)+\frac{1}{n}=\sigma(j), \\0, &\text{else},\end{array}\right.\]
corresponding to the shift by $\frac1n$ to the left on $c_0(\Q), \ell^p(\Q), 1\leq p<\infty$, respectively.
Since $a_{ii}=0$ for all $i\in\mathbb{N}$, $T$ is not positively embeddable by Theorem \ref{thm:zeros} (b).   We now show that $T$ is embeddable into a real $C_0$-group on $\ell^2$. 
We first construct a bijection $\tau:\Q\to\Z^2$ as follows. First define $\tau:\Q\cap[0,1)\to\{0\}\times\Z$ to be any bijection and then extend it to $\Q$ using the rule 
$$
\tau(q+1)=\tau(q)+(1,0) \quad \forall q\in\Q. 
$$
In this way, $T$ is real-isomorphic to the left shift operator $L$ on $\ell^2(\Z^2)$ given by 
$$
 L(e_{j+1}(n)):=e_{j}(n),\quad j,n\in\Z,
$$
where $(e_j(n))_{j,n\in\Z}$ is the canonical basis of $\ell^2(\Z^2)$.
By Example \ref{ex:shift} below, $L$ (and hence $T$) is embeddable into a real $C_0$-group.
\end{ex}

\begin{corollary}\label{cor:triangular}
If a positive (finite or infinite) matrix is upper triangular and positively embeddable into a $C_0$-semigroup, then each operator in this semigroup is  necessarily upper triangular. 
\end{corollary}
The proof follows immediately from Theorem \ref{thm:zeros} (b) and the fact that the product of two upper triangular matrices is again upper triangular. Note that in this case the generator is also upper triangular whenever it is bounded.

Analogously, we obtain the following. In our setting, a positive (finite or infinite) matrix $T$ is reducible if and only if there exists a family $(e_j)_{j\in J}$ of basis vectors such that the subspace $Y:=\ol{\lin}\{e_j,j\in J\}$ is nontrivial and satisfies $TY\subset Y$. 

\begin{corollary}\label{cor:red}
Let $T(\cdot)$ be a positive (finite or infinite-dimensional) matrix semigroup. If $T(1)$ is reducible, then for every $t>0$ the matrix $T(t)$  is reducible with the same reducing subspaces. 
\end{corollary}

\begin{corollary}
Let $T$ be a block matrix with square blocks on the diagonal. Then every positive semigroup into which $T$ can be embedded is of the same form, as well as the generator whenever it is bounded. In particular, each block of $T$ is embeddable into a positive semigroup.
\end{corollary}

\begin{corollary}\label{cor:id-pos}
The (finite or infinite) identity matrix $I$ is uniquely positively embeddable into the group given by $T(t)=I$.
\end{corollary}
\begin{remark}
The identity is however embeddable into many real semigroups. Indeed, for instance every rotation on $\R^d$, $d\geq 2$ leads to a (finite-dimensional) real rotation semigroup  into which the identity can be embedded (after rescaling if necessary), and it is easy to construct from such blocks many infinite matrix semigroups into which the identity can be embedded. 
\end{remark}

\subsection{Embeddability into positive analytic semigroups}
Based on the previous subsection we give a necessary condition for embeddability of an operator 
into a positive analytic semigroup and describe the zero pattern of the semigroup. 
The following result can be found in a more general form in \cite[Thm.~C.III.3.2(b)]{book-pos}, see also
\cite[Exercise 7.7.2]{BKR} for its finite-dimensional version. 
We present a direct proof for the reader's convenience.

\begin{thm}\label{thm:strict-pos}
Let a (finite or infinite) matrix $T$ be embeddable into a positive analytic semigroup $T(\cdot)$. 
Then $T$ is either strictly positive or reducible. Moreover, exactly one of the following assertions holds.
\begin{enumerate}  
\item $T(t)$ is strictly positive for every $t>0$. 
\item $T(t)$ is reducible for every $t>0$ with the same reducing subspaces as $T$. 
\end{enumerate}  
\end{thm}
\begin{proof}
By Theorem \ref{thm:zeros} (a), the diagonal entries of $T(t)$ are strictly positive for every $t\geq 0$. 
If $T$ is irreducible and $a_{ij}=0$ for some $i\neq j$, 
then $a_{ij}(t)=0$ for every $t\in[0,\infty)$ by Theorem \ref{thm:zeros} (b). 
But by \cite[Prop.\ III.8.3]{schaefer}, there exists $k\in\N$ such that $a_{ij}(k)\neq 0$ which is a contradiction. 
So $T$ is strictly positive. In this case, every $T(t)$, $t>0$, is strictly positive by rescaling and Theorem \ref{thm:zeros} (b), hence (a) holds.

If $T$ is reducible, then we can apply Corollary~\ref{cor:red} to show assertion (b). 
\end{proof}

\begin{remark}
Note that the bounded generator of a strictly positive semigroup does not have to be strictly positive outside of the diagonal as the example 
$
A=\left(
\begin{matrix}
0 & 1 & 0\\
0 & 0 & 1\\
1 & 0 & 0
\end{matrix}
\right)
$ 
shows. (Note that $A$ and the generated semigroup are circulant. For results on this class of matrices in the context of Markov embeddings see \cite[Section 5]{BS}.)
\end{remark}

The converse implication fails as the  following examples show. More examples will be provided in Section \ref{sec:2x2}.

\begin{ex}[Jordan blocks are  positively embeddable if and only if $d\leq 2$]\label{ex:pos-Jordan}

Let $J$ be a positive $d$-dimensional Jordan block corresponding to $\lambda>0$.
Note that $J$ is reducible 
for $d\geq 2$ and embeddable into a real semigroup, see  
Example \ref{ex:jordan-real}.  For $d\in\{1,2\}$ $J$ is positively embeddable into 
 $T(\cdot)$ with
$T(t)=\lambda^t$ and $T(t)=\lambda^{t}\left(\begin{matrix}
1 & \frac{t}{\lambda} \\
0 & 1 
\end{matrix}\right)$, respectively.
We show that for $d\geq 3$, $J$ 
is not positively embeddable. 

Let $d\geq 3$, assume that $J$ is positively embeddable into $T(\cdot)$ and consider $S:=T(1/2)=(b_{ij})$. By $(S^2)_{13}=J_{13}=0$ we have $b_{12}b_{23}=0$. Since, by Theorem \ref{thm:zeros}, $S$ has zero entries at the places where $J$ has zero entries, we have  either $(S^2)_{12}=\sum_{j=1}^d b_{1j}b_{j2}=0$ or $(S^2)_{23}=0$, both contradicting $S^2=J$.
\end{ex}
\begin{ex}[Infinite-dimensional Jordan block]
Also the infinite-dimensional Jordan block
$$
J=\left(\begin{matrix}
1 & 1 & 0 & 0 &\ldots\\
0 & 1  & 1 & 0 &\ldots\\
0 & 0 & 1 & 1 & \ldots\\
\  &\ & \ldots
\end{matrix}\right)
$$
 is not positively embeddable. Indeed, assume it is embeddable into a positive semigroup $T(\cdot)$. By Corollary \ref{cor:red}, the $3$-dimensional Jordan block would be embeddable into a positive semigroup, a contradiction to the previous example.
\end{ex}

\section{$2\times 2$ case}\label{sec:2x2}

The following provides a full description of positive embeddability in dimension $2$. 
An analogous result for Markov embeddability can be found in \cite{BS} while the equivalence of (ii) and (iii) is stated in \cite[Thm.~3.1]{tam-huang}.

\begin{thm}\label{thm:2x2}
  Let $T\in\R^{2\times 2}$ be positive.
 Then the following assertions are equivalent.
\begin{enumerate}
\item[(i)] $T$ is embeddable into a positive semigroup.
\item[(ii)] $T$ is invertible and has a positive square root.
\item[(iii)] $\sigma(T)\subseteq (0,\infty)$. 
\item[(iv)] $\det T>0$.
\end{enumerate}
\end{thm}

\begin{proof}
We begin with an observation. 
Since $T$ is positive, the spectral radius $\radius(T)\geq 0$ is an eigenvalue. 
  If $\lambda$ is an eigenvalue of a real matrix, then so is $\overline{\lambda}$. Hence, a positive $2\times 2$ matrix has only real eigenvalues, one of which is $r(T)\geq 0$.

We now proceed with the proof of the assertion. Clearly, (i) implies (ii).

(ii)$\Rightarrow$(iii)
Let $\lambda\in\R\setminus\{0\}$ be the second eigenvalue of $T$.
Since $T$ has in particular a real square root, $\lambda>0$ by the discussion at the beginning of Subsection \ref{subsec:real-finite} and the observation above.

(iii)$\Leftrightarrow$(iv) follows from the observation and the fact that the determinant of $T$ equals the product of the eigenvalues.

(iii)$\Rightarrow$(i)
  Case 1: $T$ has two distinct eigenvalues $0<\mu<\lambda$.
  
  By the theory of positive matrices there is a positive eigenvector $u:=(x,y)^T>0$ corresponding to the eigenvalue $\lambda$.
  Without loss of generality let $x>0$.
  Since only eigenvectors corresponding to $\lambda$ can be positive, every real eigenvector pertaining to $\mu$ has two non-zero entries with different signs.
  So there exists $z>0$ such that $v:=(x,-z)^T$ is an eigenvector pertaining to $\mu$.
  For any $t\geq 0$ we define 
  \[T(t)u:=e^{t\log(\lambda)}u \quad\text{ and }\quad T(t)v:=e^{t\log(\mu)}v.\]
  Clearly, $(T(t))_{t\geq 0}$ is a semigroup. 
  Moreover, we have 
  \[\begin{pmatrix}1\\0\end{pmatrix}=\frac{z}{x(y+z)}u+\frac{y}{x(y+z)}v \quad \text{ and } \quad \begin{pmatrix}0\\1\end{pmatrix}=\frac{1}{y+z}u-\frac{1}{y+z}v. \]
  Hence, we obtain
  \[x(y+z)T(t)\begin{pmatrix}1\\0\end{pmatrix}=\begin{pmatrix}x(ze^{t\log(\lambda)}+ye^{t\log(\mu)})\\zy(e^{t\log(\lambda)}-e^{t\log(\mu)})\end{pmatrix}\geq0\]
  and
  \[(y+z)T(t)\begin{pmatrix}0\\1\end{pmatrix}=\begin{pmatrix}x(e^{t\log(\lambda)}-e^{t\log(\mu)})\\ye^{t\log(\lambda)}+ze^{t\log(\mu)}\end{pmatrix}\geq 0\]
  showing that $T(t)\geq 0$ for all $t\geq 0$. 

  Case 2: $\sigma(T)=\{\radius(T)\}$, $T$ is diagonalisable. 

  Since $T$ is diagonalisable, there exists an invertible matrix $S\in\R^{2\times 2}$ such that  $T=S^{-1}\begin{pmatrix}\radius(T) & 0 \\ 0 & r(T) \end{pmatrix}S$ showing that $T=\radius(T)\Id$. 
  Clearly, $T$ is embeddable into the positive semigroup $(T(t))_{t\geq 0}$ 
given by $T(t)=e^{t\log(\radius(T))}\Id$.

Case 3: $\sigma(T)=\{\radius(T)\}$, $T$ is not diagonalisable.

  We claim that in this case $T$ is a triangular matrix. In fact, if all entries of $T$ were greater than $0$, then $T$ would be irreducible and $r:=\radius(T)$ would be a simple root of the characteristic equation, see  \cite[Prop.~I.6.3]{schaefer} in contradiction to our assumption that $\sigma(T)=\{r\}$. If we assumed that a diagonal entry of $T$ was $0$, e.g.~$T=\begin{pmatrix}a & b \\ c & 0\end{pmatrix}$, $a, b, c\geq 0$, then $\det T=-bc\leq 0$ contradicting the already proven equivalence (iii)$\Leftrightarrow$(iv).

    So let us assume without loss of generality that $T=\begin{pmatrix}r & b \\ 0 &r \end{pmatrix}$.
    Then the semigroup $(T(t))_{t\geq 0}$ with 
    \[T(t)=\begin{pmatrix}e^{t\log (r)} & t\frac{b}{r}e^{t\log(r)}\\ 0 & e^{t\log(r)}\end{pmatrix}\]
    is positive and $T(1)=T$. 
  
\end{proof}

\begin{remark}
  Analysing the proof of Theorem~\ref{thm:2x2} (and using Corollary \ref{cor:id-pos}) we see that the embedding into a positive 
 semigroup is unique. This is consistent with the uniqueness for $2\times 2$ matrices in the context of Markov embeddings  \cite[Cor.~3.2]{BS}.
Moreover, if $T$ is a  $2\times 2$ Markov matrix then each of the conditions in Theorem~\ref{thm:2x2} is equivalent to $T$ being embeddable into a Markov semigroup \cite[Thm.~3.1]{BS}.
\end{remark}

\begin{remark}
By Theorem \ref{thm:real} and Proposition \ref{prop:spectrum-zero}, a matrix $T\in\R^{2\times 2}$ is real-embeddable if and only if it is invertible and has no negative eigenvalue with geometric multiplicity $1$. Thus, for positive $2\times 2$ matrices real-embeddability coincides with positive embeddability by Theorem \ref{thm:2x2} since $r(T)\in\sigma(T)$ for positive $T$.  
\end{remark}

\begin{remark}
The equivalence (i)$\Leftrightarrow$(iv) in Theorem \ref{thm:2x2} (or Theorem \ref{thm:strict-pos} for finite matrices) shows in particular that a positive additive perturbation of a positively embeddable matrix does not have to be positively embeddable. 
\end{remark}

The following example shows that the implications (ii)$\Rightarrow$(i), (iii)$\Rightarrow$(ii), and (iv)$\Rightarrow$(ii)  in Theorem \ref{thm:2x2} fail for $3\times 3$ matrices. 
\begin{ex} A positive matrix
$$
T=
\left(\begin{matrix}
1 & a &c \\
0 & 1 & b\\
0 & 0 & 1
\end{matrix}\right)
$$
is positively embeddable if and only if $c\geq \frac{ab}2$, whereas $T$ has a positive square root whenever $c\geq \frac{ab}4$. Indeed, if $T$ is positively embeddable, then  by Corollary \ref{cor:triangular} the generator is  necessarily of the form $A=\left(\begin{matrix}
0 & \alpha &\gamma \\
0 & 0 & \beta \\
0 & 0 & 0
\end{matrix}\right)$ for $\alpha,\beta,\gamma\geq 0$, and every such $A$ generates the positive (semi)group given by 
$$
T(t)=
\left(\begin{matrix}
1 & t\alpha &t\gamma+\frac{t^2}2\alpha\beta \\
0 & 1 & t\beta\\
0 & 0 & 1
\end{matrix}\right).
$$
Taking $t=1$ and comparing the entries implies the first assertion. For the second assertion, let  
$$
S=\left(\begin{matrix}
1 & \alpha &\gamma \\
0 & 1 & \beta \\
0 & 0 & 1
\end{matrix}\right)
$$
be a positive matrix.  Since $S^2=\left(\begin{matrix}
1 & 2\alpha &2\gamma + \alpha\beta\\
0 & 1& 2\beta \\
0 & 0 & 1
\end{matrix}\right)$, we see that $S^2=T$ holds if and only if $\alpha=\frac{a}2$, $\beta=\frac{b}2$, $\gamma=\frac{c-ab/4}2$. Therefore $T$ has a positive triangular square root  if and only if $c\geq \frac{ab}4$.
\end{ex}

\section{Real embeddability is typical for infinite matrices} \label{sec:typical}

In this section we show that real embeddability is typical for real contractions on $\ell^2$ with respect to the strong operator topology.

We first recall some topological terminology, see, e.g., Kechris \cite[Chapter 8]{kechris}. Let $E$ be a topological space. A set $A\subset E$ is called \emph{nowhere dense} if for every  non-empty open set $U\subset E$ there exists a non-empty open set $V\subset U$ with $A\cap V=\emptyset$. Furthermore, $A$ is called \emph{meager} (or  \emph{of first category}) if it can be written as a countable union of nowhere dense sets, and \emph{residual} if its complement is meager. For a space satisfying the Baire category theorem (in particular, for a complete metric space), we say that a property is \emph{typical} if the set of points satisfying it forms a residual subset of $E$. 

Since $\ell^2$ is isomorphic via a real isometric isomorphism to 
$$
\ell^2(\N^2)=\left\{(x_{j,n})_{j,n\in\N}\in \C^{\N^2}:\, \sum_{j,n=1}^\infty|x_{j,n}|^2<\infty\right\},
$$ we consider $X:=\ell^2(\N^2)$. We are interested in the space 
$$
 E:=\{T\in\mathcal{L}(X): T \text{ is a real contraction on }X\}
$$
endowed with the strong operator topology 
and the set
$$
M:=\{T\in E: T \text{ is embeddable into a real contractive $C_0$-semigroup on }X\}.
$$
Note that $E$ is a complete metric space with respect to the metric 
$$
d(T,S)=\sum_{n,j=1}^\infty\frac{\|T(e_{j}(n))-S(e_{j}(n))\|}{2^{n+j}},
$$
where $(e_j(n))_{j,n\in \N}$ denotes the canonical basis of $X$. 

We begin with the following important example. 

\begin{ex}\label{ex:shift}
Consider the \emph{(infinite-dimensional backward unilateral) shift} $L:\ell^2(\N^2)\to \ell^2(\N^2)$ defined by
$$
L(e_1(n))=0,\ L(e_{j+1}(n)):=e_{j}(n),\quad j,n\in\N.
$$
We show that $L\in M$. It is clear that $L$ is a real contraction. We adapt the construction of the semigroup from \cite[Prop.~4.3]{eis}. First observe that $\ell^2(\N^2)$ is  unitarily isomorphic to $\ell^2(\N,\ell^2)$ via the real isomorphism $(x_{j,n})_{j,n\in\N}\mapsto (x_{j,\cdot})_{j\in\N}$. (Note that the inverse of a real operator is automatically real.) Furthermore, $\ell^2(\N,\ell^2)$ is uni\-tarily isomorphic to $\ell^2(\N,L^2[0,1])$ via the real unitary correspondence mapping the canonical basis of $\ell^2$ onto the classical real trigonometric orthonormal basis of $L^2[0,1]$. Finally, $\ell^2(\N,L^2[0,1])$ is  unitarily isomorphic to $L^2[0,\infty)$  by virtue of  the real isomorphism
$$
(f_1,f_2,\ldots)\mapsto (s\mapsto f_{n-1}(s-n+1),\ s\in [n-1,n)).
$$
The operator on $L^2[0,\infty)$ corresponding to $L$ is exactly the left shift given by
$$
(Tf)(s):=f(s+1), \quad s\geq 0.
$$
This operator is clearly embeddable into the real (even positive) shift semigroup given by $(T(t)f)(s):=f(s+t)$, $s,t\geq 0$. Analogously, the shift on $\ell^2(\Z^2)$ is real-isomorphic to the left shift on $L^2(\R)$ and hence is embeddable into a real $C_0$-group.

Note that $L$ is not positively embeddable because the corresponding operator on $\ell^2$ (obtained by a natural positive isomorphism between $\ell^2(\N^2)$ and $\ell^2$) is not positively embeddable by Theorem \ref{thm:zeros}. An analogous assertion holds for the shift on $\ell^2(\Z^2)$.
\end{ex}

The following is the main result of this section.

\begin{thm}[Real-embeddability is typical]\label{thm:shift}
The conjugacy class  of the infinite-dimensional shift
$$
O(L):=\{S^{-1}LS:\, S\text{ real unitary}\}
$$ 
is a residual
subset of $E$ with $O(L)\subset M$. In particular, $M$ is residual in $E$.
\end{thm}
\begin{proof}
The inclusion $O(L)\subset M$ follows from Example \ref{ex:shift}: If $L$ is embeddable into a real contractive $C_0$-semigroup $L(\cdot)$, then $S^{-1}L S$ is embeddable into the $C_0$-semigroup $S^{-1}L(\cdot) S$, which is again real and contractive.

To prove that $O(L)$ is a 
residual subset of $E$ we use the real version of \cite[Theorem 5.2]{EM} stating that  for the real space $X_{\R}:=\ell^2_\R(\N^2)$, the conjugacy class of the infinite-dimensional backward shift operator $L_\R$
$$
O_\R(L):=\{S^{-1}L_\R S:\ S \text{ unitary on }X_\R\}
$$
is a residual subset of the space $C(X_\R)$ of contractions on $X_\R$. Note that the proof of \cite[Theorem 5.2]{EM} is valid for both the real and the complex space.

We now use that $X$ is the complexification of $X_\R$. Recall that
for a real ope\-rator $T$ on $X$ with 
restriction $T_\R$ on $X_\R$, $\|T\|=\|T_\R\|$ holds, see, e.g.,  \cite[Prop.~4]{MST}.

Consider now the isometric isomorphism $\pi:E\to C(X_\R)$ given by 
$$
\pi(T):=T_\R.
$$
We see that 
$\pi(O(L))=O_\R(L)$ by the definition of $O(L)$. Since $O_\R(L)$ is residual in $C(X_\R)$, $O(L)$ is residual in $E$.
\end{proof}

As the following shows, the situation is very different for finite-dimensional spaces. 

\begin{prop}[No residuality in finite dimensions]
Let $d\geq 1$ and consider the space 
 $\R^{d\times d}$ of real $d\times d$ matrices
 endowed with the norm ($=$ strong operator) topology. Then both the set of real-embeddable matrices and its complement contain a non-empty open subset. 
In particular, neither real-embeddability nor non-real-embeddability is typical. 
\end{prop}
\begin{proof}
For a matrix $T$ and an eigenvalue $\lambda$ of $T$, denote by $h(\lambda, T)$ the dimension of the corresponding generalised eigenspace.

We begin with the following simple observation. Let $(T_n)$ be a sequence in $\R^{d\times d}$
converging to $T\in 
\R^{d\times d}$. By the pigeonhole  principle, we can find a subsequence which we call $(T_n)$ again such
that each $T_n$ has the same number of different eigenvalues, say $l$. For $n\in\N$, denote the eigenvalues of $T_n$ by $\lambda_{1,n},\ldots,\lambda_{l,n}$ and by $k_j$ the dimension of the generalised eigenspace of $\lambda_{j,n}$, $j\in\{1,\ldots,l\}$ which we 
 can assume to be independent of $n$ by passing to a subsequence if necessary. Let  $p_n(z)=(z-\lambda_{1,n})^{k_1}\cdots(z-\lambda_{l,n})^{k_l}$ be the characteristic polynomial of $T_n$. We also can assume by the pigeonhole  principle that $(\lambda_{j,n})$ converges to some $\mu_j$ as $n\to\infty$ for every $j\in\{1,\ldots,l\}$. By assumption, $p_n$ converges pointwise to the characteristic polynomial $p$ of $T$ which is then of the form $p(z)=(z-\mu_{1})^{k_1}\cdots(z-\mu_{l})^{k_l}$, and $\mu_1,\ldots,\mu_l$ are (not necessarily distinct) eigenvalues of $T$. Thus the dimension of the generalised eigenspace of $\mu_j$ satisfies
\begin{equation}\label{eq:hauptraum}
h(\mu_j, T)=\sum_{i:\mu_i=\mu_j}k_i=\sum_{i:\mu_i=\mu_j}h(\mu_i,T_n).
\end{equation}

We now return to the real-embeddability problem. 
Consider the set
$$
M_1:=\{T\in \R^{d\times d}:\ \Re \mu>0\ \ \forall \mu\in\sigma(T)\}.
$$
Each $T\in M_1$ is real-embeddable by Theorem \ref{thm:real}. Moreover,
by the observation above, $M_1$ is open. Indeed, if not, then there exists $T\in M_1$ and a sequence $(T_n)$ in $\R^{d\times d}$ converging to $T$ with $T_n\notin M_1$ for every $n\in\N$. Then for every $n$ there exists an eigenvalue $\lambda_n$ of $T_n$ with $\Re \lambda_n\leq 0$. Since every limit point of $(\lambda_n)$ has to be an eigenvalue of $T$, this contradicts the assumption.

Consider now the set 
\begin{eqnarray*}
M_2:=\{T\in \R^{d\times d}:\ T\text{ has an eigenvalue $\mu$ in }(-\infty,0)\text{ with }h(\mu, T)=1\},
\end{eqnarray*}
By  Theorem \ref{thm:real}, each $T\in M_2$ is not real-embeddable. It remains to show that $M_2$ is open. Assume that this is false, i.e., there exists $T\in M_2$ and a sequence $(T_n)$ from the complement on $M_2$ which converges to $T$. By the observation above, there is a sequence $(\mu_n)$ in $\C$ with $\mu_n\in \sigma(T_n)$ which converges to $\mu\in \sigma(T)\cap (-\infty,0)$ with $h(\mu, T)=1$. By (\ref{eq:hauptraum}) and the fact that non-real eigenvalues of a real matrix come in pairs of the form $\lambda,\ol{\lambda}$, we conclude that each $\mu_n$ satisfies $\mu_n\in(-\infty,0)$ and $h(\mu_n, T_n)=1$. This contradicts the assumption.

The last assertion follows from the Baire category theorem. 
\end{proof}

Topological properties in the context of finite Markov embeddings were  investigated in, e.g., \cite[Prop.~3]{kingman}, \cite{davies},  \cite{BS}.

\vspace{0.1cm}

\textbf{Acknowledgement.} Our paper was motivated by the recent work of M.~Baake and J.~Sumner \cite{BS} on the Markov embedding problem. We are very grateful to Michael Baake for the inspiration and helpful comments,
 to Rainer Nagel for interesting discussions, and to Jochen Gl\"uck for valuable remarks and references. We sincerely thank the referees for their comments and suggestions which improved the presentation of the paper.

\end{document}